\documentclass{ifacconf}

\usepackage{amsmath}
\usepackage{amsfonts}
\usepackage{amssymb}

\usepackage{graphicx}      % include this line if your document contains figures
\usepackage{natbib}        % required for bibliography
\usepackage{url}

\usepackage[final,commentmarkup=uwave]{changes}
\definechangesauthor[name={Torbjørn}, color=purple]{TC}
\definechangesauthor[name={Ilya, color=violet}]{IK}

\usepackage[IFAC]{altthm-styles}

\DeclareMathOperator*{\argmin}{arg\,min}

\newcommand{\Def}[1]{\overset{\text{\tiny def}}{#1}}
\begin{document}
\begin{frontmatter}

\title{Input-to-State Stability of a Bilevel Proximal Gradient Descent Algorithm\thanksref{footnoteinfo}} 

\thanks[footnoteinfo]{Partially supported by United States Air Force Office of Scientific Research Grant number FA9550-20-1-0385.}%{Supported by the authors' institutions.}

\author[UMich,UStutt]{Torbjørn Cunis} and 
\author[UMich]{Ilya Kolmanovsky} 

\address[UMich]{Department of Aerospace Engineering, University of Michigan,
   Ann~Arbor, MI 48109, USA (e-mail: {\tt \{tcunis | ilya\}@umich.edu})}
\address[UStutt]{Institute of Flight Mechanics and Control, University of Stuttgart, 
   70569 Stuttgart, Germany (e-mail: {\tt tcunis@ifr.uni-stuttgart.de})}

\begin{abstract}                % Abstract of not more than 250 words.
This paper studies convergence properties of inexact iterative solution schemes for bilevel optimization problems. Bilevel optimization problems \deleted[id=TC]{often} emerge in control-aware design optimization, \added[id=TC]{where} the system design parameters are optimized in the outer loop and a discrete-time control trajectory is optimized in the inner loop,  \replaced[id=TC]{but}{Bilevel optimization problems} also \replaced[id=TC]{arise}{emerge} in other domains \added[id=TC]{including} machine learning.  In the paper an interconnection of proximal gradient algorithms is \added[id=TC]{proposed} to solve the inner loop and outer loop optimization problems in the setting of control-aware design optimization and robustness is analyzed from a control-theoretic perspective. By employing input-to-state stability arguments, conditions are derived that ensure convergence of the interconnected scheme to the optimal solution for a class of the bilevel optimization problem. 
\end{abstract}

\begin{keyword}
Input-to-state stability,
Optimal control theory,
Proximal gradient descent,
Stability of nonlinear systems,
Time-distributed optimization.
\end{keyword}

\end{frontmatter}
%===============================================================================

\section{Introduction}

In bilevel optimization \citep[see][and references therein]{colson2007overview} the cost function and/or constraints of a high-level optimization problem are informed by the optimal value function of an lower-level optimization problem which, in turn, is parametrized by the \added[id=TC]{decision variables} of the higher level optimization problem.  

Bilevel optimization has numerous  applications.   For instance, in control co-design \citep[see, e.g.,][]{garcia2019control},
plant and controller parameters need to be simultaneously optimized, while in control-aware design \citep[such as described by][]{cunis2022control}, system design parameters are optimized in the outer loop  subject to the existence of an admissible state and control trajectory determined through the inner-loop optimization.  In those works, the problem is reformulated in such a way that the gradients of the value function of the inner loop optimization problem with respect to the parameters are employed for the outer loop optimization. \Citet{cunis2022control} argued that from the perspective of the integration with existing gradient-based multi-disciplinary design optimization solvers, such an approach is preferable to simultaneous optimization with respect to design and controller parameters.  In that work, such an approach is employed for optimizing wing shape parameters subject to glide slope following maneuver constraints at landing of a supersonic aircraft.

\Citet{bennett2008bilevel} and \citet{franceschi2018bilevel} have discussed bilevel optimization problems in machine learning, where the outer-loop optimization is performed with respect to the hyperparameters of the machine learning solution.  The implementation of robust reference governors \citep[see, e.g.,][and references therein]{garone2017reference} involves outer loop optimization with respect to the reference command and the inner loop optimization with respect to the disturbance sequence aimed at maximizing constraint violation.  Finally, leader-follower and Stackelberg games lend themselves naturally to bilevel programming \citep[see, e.g.,][]{basilico2017bilevel}.

\added[id=TC]{Typical for these bilevel optimization problems is that evaluating value function and and gradient of the inner-loop optimization is computationally expensive.}
In this paper,  a prototypical problem in bilevel optimization which is typical of control-aware design is considered. Here, an optimal control problem is solved in the inner loop and a parameter optimization problem being solved in the outer loop subject to the solution of the inner loop.   The aim is to demonstrate that the exact solution of the inner-loop optimization problem \added[id=TC]{is} not necessary for the overall convergence of the bilevel optimization.  This has obvious advantages when the time to compute the solution and the available computing power are restricted.

More specifically, we focus on a class of bilevel optimization problems where a discrete-time linear quadratic optimal control problem \added[id=TC]{with input constraints} is solved in the inner loop and both the linear model and the quadratic cost function depend on the parameter of the outer-loop optimization problem. The optimal value function of the inner loop optimization problem, which is thus a function of the parameter, \added[id=TC]{is to be minimized subject to additional constraints on the parameter}.  Within this setting, we propose a framework for the analysis of \deleted[id=TC]{inexact} bilevel optimization
based on \added[id=TC]{inexact} proximal gradient algorithms, where the number of iterations of the inner-loop optimizer remains fixed.  Thus, the gradient of the inner-loop value function computed only approximately is used for a proximal gradient descent with respect to the outer-loop \added[id=TC]{decision variables}.  By exploiting the control-theoretic \added[id=TC]{notion} of input-to-state stability (ISS) and a small gain theorem in the analysis, we derive sufficient
conditions for asymptotic stability and convergence of the overall bilevel optimization algorithm.

%Contributions:
%- Framework for analysis of multilevel optimization, where the number of iterations of the inner loop are fixed.
%- Proof of parameter-independent ISS for parametrized projected gradient descent with nonunique reference parameter
%
%- Proof of ISS for projected (proximal) gradient descent with perturbed gradients

\paragraph*{Notation}
$\mathbb N$ and $\mathbb R$ are the sets of natural and real numbers, respectively. The extended real numbers are $\overline {\mathbb R} = \mathbb R \cup \{\infty\}$. A convex function $f: \mathbb R^n \to \overline {\mathbb R}$ is proper if and only if its domain, $\{ x \in \mathbb R^n \, | \, f(x) \leq \infty \}$, is nonempty. For some $m \in \mathbb N$, the set of positive semi-definite (resp., positive definite) matrices in $\mathbb R^{m \times m}$ is denoted by $\mathbb S_m^+$ (resp., $\mathbb S_m^{++}$). The Euclidean norm and inner product on $\mathbb R^n$ are denoted by $| \cdot |$ and $\langle \cdot, \cdot \rangle$. For a function $\alpha: \mathbb R_{\geq 0} \to \mathbb R_{\geq 0}$ we write that $\alpha \in \mathcal K$ if and only if $\alpha$ is continuous, positive definite, and strictly increasing; and that $\alpha \in \mathcal K_\infty$ if and only if $\alpha \in \mathcal K$ as well as $\alpha(r) \to \infty$ if $r \to \infty$. Moreover, $\mathcal {KL}$ is the class of continuous functions $\beta: \mathbb R_{\geq 0} \times \mathbb R_{\geq 0} \to \mathbb R_{\geq 0}$ satisfying $\beta(\cdot, s) \in \mathcal K$ for all $s \geq 0$ and $\beta(r, \cdot)$ is decreasing with $\lim_{s \to \infty} \beta(r, s) = 0$ for all $r \geq 0$.

%\newpage
\section{Preliminaries}
Given a proper convex and lower semicontinuous function $f: \mathbb R^n \to \overline {\mathbb R}$, %and matrix $U \in \mathbb S_n^{++}$, 
the proximity operator of $f$ maps $x$ to the unique solution of
\begin{align*}
    \operatorname{prox}_f(x) = \argmin_{y \in \mathbb R^n} f(y) + \frac{1}{2} | x - y |^2
\end{align*}
%with $\| y \|_U^2 \Def= y^\top U y$ for $y \in \mathbb R^n$. 
for all $x \in \mathbb R^n$. 
If $f$ is the indicator function of a convex set $\mathcal C \subset \mathbb R^n$, the proximity operator of $f$ corresponds to a projection onto~$\mathcal C$. %with respect to the distance $\| \cdot \|_U$ and is denoted by $\operatorname{prox}_C^U$. We omit the superscript if $U$ is the identity on $\mathbb R^{n \times n}$.

We make extensive use of the fact that the proximity operator is nonexpansive as detailed in the next statement. Further properties can be found in the appendix.

\begin{property}[\citealp{Rockafellar1976}]
    \label{prop:prox-nonexp}
    Let $f: \mathbb R^n \to \overline {\mathbb R}$ be a proper convex and lower semicontinuous function and take $\nu > 0$; the proximity operator satisfies
    \begin{align*}
        \big| \operatorname{prox}_{\nu f}(x_1) - \operatorname{prox}_{\nu f}(x_2) \big| \leq | x_1 - x_2 |
    \end{align*}
    for all $x_1, x_2 \in \mathbb R^n$.
\end{property}

\subsection{Problem Statement}
Consider a system represented by a parameter-dependent model of the form
\begin{align}
    \label{eq:linear-system}
    % \Sigma(x,\vec u,p): \quad \left\{
    \begin{aligned}
        x_{k+1} &= A(p) x_k + B(p) u_k + e(p)
        % x_0 &= x
    \end{aligned}
    % \right.
\end{align}
with initial condition $x_0 \in \mathbb R^n$, control input $u_k \in \mathbb R^m$,   $k \in \mathbb N$, and parameter $p \in \mathbb R^l$; the functions $A: \mathbb R^l \to \mathbb R^{n \times n}$, $B: \mathbb R^l \to \mathbb R^{n \times m}$, and $e: \mathbb R^l \to \mathbb R^n$ are smooth. 

We define the finite-horizon quadratic cost
\begin{align}
	\label{eq:quadratic-cost}
    J &= \frac{1}{2} x_N^\top S(p) x_N + \frac{1}{2} \sum_{k=0}^{N-1} \big( x_k^\top Q(p) x_k + u_k^\top R(p) u_k \big),
\end{align}
where $N \in \mathbb N$ is the horizon, which is evaluated on the trajectories of (\ref{eq:linear-system}) for a specified initial condition $x_0$.
%where $x_{k+1} = A(p) x_k + B(p) u_k + e(p)$ for all $k \in \{0, \ldots, N-1\}$ and some suitable initial condition $x_0$. 
We assume that $R: \mathbb R^l \to \mathbb S_m^{++}$ and $Q, S: \mathbb R^l \to \mathbb S_n^+$ are smooth. Collecting $\mathbf u = (u_0^\top, \ldots, u_{N-1}^\top)^\top$, the state sequence $\mathbf x = (x_1^\top, \ldots, x_N^\top)^\top$ satisfies
\begin{align*}
    {\mathbf x} = \hat A(p) x_0 + \hat B(p) \mathbf u + \hat e(p)
\end{align*}
for some constant $x_0$, where $\hat A(p)$ and $\hat B(p)$ are matrix-valued and $\hat e(p)$ is vector valued; and 
% If $x_0$ is constant, 
the cost $J$ is a quadratic function in $\mathbf u$ and smooth in $p$, viz.
\begin{align}\label{equ:costf}
    J(\mathbf u, p) = \mathbf u^\top H(p) \mathbf u + g(p)^\top \mathbf u + c(p)
\end{align}
where $H(p)$ is matrix-valued, $g(p)$ is vector-valued, and $c(p)$ is scalar; all functions are smooth in $p$.  Note that, under our assumptions, $H(p)$ is  positive-definite.

Suppose $P: \mathbb R^l \to \overline {\mathbb R}$ and $U: \mathbb R^{Nm} \to \overline {\mathbb R}$ are proper convex and lower semicontinuous penalty functions for $p$ and $\mathbf u$, respectively, with compact domains $\mathcal P$ and $\mathcal U$.
Given some $p \in \mathbb R^l$, the cost function $J(p, \mathbf u) + U(\mathbf u)$ is strongly convex in $\mathbf u$ with optimal value $\bar J: p \mapsto \min_{\mathbf u} J(\mathbf u, p) + U(\mathbf u)$ and thus has a unique optimal solution $\bar {\mathbf u}(p)$. 
%Suppose the parameter $p$ is constrained to a compact set, $p \in P \subset \mathbb R^l$; states and inputs are constrained to $\mathcal X = X_1 \times \cdots \times X_N$ and $\mathcal U = U_0 \times \cdots \times U_{N-1}$, respectively, where $X_k \subset \mathbb R^n$ and $U_k \subset \mathbb R^m$ are convex sets for all $k \in \{0, \ldots, N-1\}$.
We consider the following optimization problem.

\begin{problem}
    Minimize $\bar J(p) + P(p)$ for $p \in \mathbb R^l$.
\end{problem}

We adopt the \added[id=TC]{bilevel} proximal gradient descent algorithm
\begin{subequations}
    \label{eq:dynopt}
\begin{align}
    \label{eq:dynopt-outer}
    p^{\ell+1} &= \operatorname{prox}_P \big[ p^\ell - \nu \widehat {\nabla_p J}(\mathbf u^{\ell+1}, p^\ell)^\top \big] \\
    \mathbf u^{\ell+1} &= \operatorname{Arg\,min}_\kappa J(\cdot, p^\ell) + U(\cdot)
\end{align}
\end{subequations}
with $\kappa, \ell \in \mathbb N$ and $\nu > 0$, where $\operatorname{Arg\,min}_\kappa$ denotes a $\kappa$-step suboptimal solution to the parametrized optimal control problem %of $J(\mathbf u, p)$ 
and $\widehat {\nabla_p J}$ denotes the vector of estimated partial derivatives of $\bar J$ at $p^\ell$ based on the suboptimal solution $\mathbf u^{\ell+1}$.

Without state constraints, the optimal control solution can be approximated by the \added[id=TC]{inner-loop} finite-horizon proximal gradient descent iteration
\begin{subequations}
    \label{eq:dynopt-inner}
\begin{align}
    \mathbf u^\ell_{k+1} &= \operatorname{prox}_{U} \big[ \mathbf u^\ell_k - \mu \nabla_{\mathbf u} J(\mathbf u^\ell_k, p^\ell)^\top \big] \\
    \mathbf u^{\ell+1}_0 &= \mathbf u^\ell_\kappa
\end{align}
\end{subequations}
for all $k \in \{0, \ldots, \kappa-1\}$ and with $\mu > 0$. 

% \begin{assumption}
%     $\bar J$ is strongly convex on $P$
% \end{assumption}

% The assumption is satisfied if, for example, the system~\eqref{eq:system} is linear parameter-varying. Under this assumption, $\bar J$ has a unique minimizer $p_\star$ and we denote $\bar {\mathbf u}_\star = \bar {\mathbf u}(p_\star)$. 

Denote by $\mathcal P_\star \subset \mathcal P$ the set of minima of $\bar J + P$.
We aim to prove a sufficient condition on $\mu$ and $\nu$ as well as $\kappa$ such that
\begin{align}
    p^\ell \to \mathcal P_\star \quad \text{and} \quad \mathbf u_\ell^\kappa \to \bar {\mathbf u}(p^\ell)
\end{align}
as $\ell \to \infty$.

The dynamics in \eqref{eq:dynopt} can be viewed as an interconnection of discrete-time systems, where the state of one system -- here, the parameter $p$ and current solution $\mathbf u$ -- enters as disturbance into the other -- through the estimated partial derivative vector and parametrized optimization, respectively. To describe a stable behaviour of the interconnection, we make use of the notion of {input-to-state stability} (ISS) firstly introduced by \citet{Sontag1995a}.
As $\mathcal P_\star$ may consist of more than one element, we rely on a since developed generalization called $\omega$ISS.

\subsection{Size functions}
Let $\mathbb X \subseteq \mathbb R^n$ be an open set and $\mathcal A \subset \mathbb X$ be compact. A {\em measurement function} (on $\mathbb X)$ is a continuous function $\omega: \mathbb X \to \mathbb R_{\geq 0}$ that is positive semidefinite.

\begin{definition}
    % \label{ass:sizefunction}
    A measurement function $\omega$ on $\mathbb X$ is a {\em size function} for $\mathcal A$ if and only if $\omega$ vanishes exactly on $\mathcal A$ as well as $\omega(\xi_k) \to \infty$ for any sequence $\{\xi_k\}_{k \geq 0} \subset \mathbb X$, if either $\xi_k \to \partial \mathbb X$ or $| \xi_k | \to \infty$.
\end{definition}

Size functions have also been called {\em proper indicators} \citep{Kellett2012,Noroozi2018}. If $\mathbb X$ is the full space, then the distance mapping %with respect to $\mathcal A$, viz.
\begin{align*}
    \operatorname{dist}_{\mathcal A}(x) \Def= \min_{\xi \in \mathcal A} | x - \xi |
\end{align*}
is a simple example for a size function.
The following result allows us to compare size functions and to establish a notion of equivalence.

\begin{lemma}[\citealp{Sontag2022}, Corollary~2.7]
    \label{lem:size-comparison}
    Let $\omega_1: \mathbb X \to \mathbb R_{\geq 0}$ be a size function for $\mathcal A$; then any measurement function $\omega_2$ on $\mathbb X$ is a size function (for $\mathcal A$) if and only if there exists $\alpha_1, \alpha_2 \in \mathcal K_\infty$ satisfying
    \begin{align*}
        \alpha_1(\omega_1(x)) \leq \omega_2(x) \leq \alpha_2(\omega_1(x))
    \end{align*}
    for all $x \in \mathbb X$.
    $\lhd$
\end{lemma}

\subsection{$\omega$-Input-to-State Stability}
We consider the nonlinear, nonautonomous system on $\mathbb X$
\begin{align}
    \label{eq:system}
    x_{k+1} = f(x_k, u_k) \tag{$\Sigma$}
\end{align}
with $k \geq 0$, where $f: \mathbb X \times \mathbb R^m \to \mathbb X$ is continuous.

\begin{definition}
    \label{def:iss}
    Let $\omega$ be a measurement function on $\mathbb X$; the system \eqref{eq:system} is {\em $\omega$-input-to-state stable} if and only if there exists $\beta \in \mathcal {KL}$ and $\gamma \in \mathcal K$ such that for all $x_0 \in \mathbb X$ and $\{u_k\}_{k \geq 0} \in \ell_m^\infty$ the solution $\{x_k\}_{k \geq 0}$ satisfies
    \begin{align*}
        \omega( x_k ) \leq \max \{ \beta(\omega(x_0), k), \gamma(\|\{u_k\}\|_\infty) \}
    \end{align*}
    for all $k \geq 0$, where $\|\{u_k\}\|_\infty \Def= \sup_{k\geq0} |u_k|$.
\end{definition}

If $\omega$ is a size function, then Definition~\ref{def:iss} reduces to the classical definition of input-to-state stability (for $\mathcal A$ on $\mathbb X$) as introduced by \citet{Sontag1995b}. However, it also generalizes notions such as state-independent input-to-output stability \citep{Kellett2012}.
As for classical input-to-state stability, system \eqref{eq:system} is $\omega$-input-to-state stable if (and only if) there exists a continuous function $V: \mathbb X \to \mathbb R_{\geq 0}$, called $\omega$ISS Lyapunov function, and gains $\alpha_1, \alpha_2, \alpha_3 \in \mathcal K_\infty$ and $\gamma \in \mathcal K$ %$\rho, \chi \in \mathcal K$ 
satisfying $\alpha_3 < \operatorname {id}$ and
\begin{subequations}
\begin{gather}
    \alpha_1(\omega(x)) \leq V(x) \leq \alpha_2(\omega(x)) \\
    % \omega(x) \geq \chi(|u|) \Rightarrow V(f(x, u)) - V(x) \leq -\rho(V(x))
    V(f(x,u)) \leq \max \{ \alpha_3(V(x)), \gamma(|u|) \}
\end{gather}
\end{subequations}
for all $x \in \mathbb X$ and $u \in \mathbb R^m$ \citep[Theorem~7]{Noroozi2018}. %\citep[Theorem~1]{Kellett2012}. 

\begin{remark}
    If $\omega$ is a size function, the second condition for an $\omega$ISS Lyapunov function can equivalently be written as an implication
    \begin{align*}
        \omega(x) \geq \chi(|u|) \Rightarrow V(f(x, u)) - V(x) \leq -\rho(V(x))
    \end{align*}
    where $\rho$ is a positive definite function and $\chi \in \mathcal K$.
    In particular, the ISS gain is
    \begin{align*}
        \gamma(r) = \max\{ V(f(x,u)) \, | \, V(x) \leq \chi(r), |u| \leq r \}
    \end{align*}
for all $r \geq 0$.
\end{remark}

We now consider an interconnection of $\wp \in \mathbb N$ nonlinear systems
%\begin{subequations}
%    \label{eq:interconnection}
\begin{align}
    \label{eq:subsystem}
    x_{i(k+1)} = f_i(x_{1k}, \ldots, x_{Rk}, u_{ik}) \tag{$\Sigma_i$}
\end{align}
%\end{subequations}
with $x = (x_1, \ldots, x_\wp) \in \mathbb X$ and $u_i \in \mathbb R^{m_i}$, where $f_i: \mathbb X \times \mathbb R^{m_i} \to \mathbb X_i$ are continuous functions, for all $i \in I = \{1, \ldots, \wp\}$. Let $\| \cdot \|$ be a monotonic norm on $\mathbb R^\wp$.

\begin{theorem}[\citealp{Noroozi2018}, Theorem~10]
    \label{thm:smallgain}
    Suppose, \\ for any $i \in I$,
    that $\omega_i$ is a measurement function on $\mathbb X_i$ and the subsystem \eqref{eq:subsystem} 
    % in \eqref{eq:interconnection} are $\omega_{1,2}-$input-to-state stable, that is,
    has an $\omega_i$ISS Lyapunov function $W_i: \mathbb X_i \to \mathbb R_{\geq 0}$ and gains $\gamma_{i1}^x, \ldots, \gamma_{i\wp}^x, \gamma_i^u \in \mathcal K \cup \{0\}$, such that $\gamma_{ii}^x < \operatorname{id}$ and
    \begin{align*}
        W_i(f_i(x, u_i)) \leq \max_{j \in I} \{ \gamma_{ij}^x(W_j(x_j)), \gamma_i^u(| u_i |) \}
    \end{align*}
    % with $\beta_{1,2} \in \mathcal {KL}$ and $\gamma_{1,2}^x, \gamma_{1,2}^u \in \mathcal K$; 
    for all $(x, u_i) \in \mathbb X \times R^{m_i}$;
    define
    \begin{align*}
        \hat \omega: (x_1, \ldots, x_\wp) \mapsto \| (\omega_1(x_1), \ldots, \omega_\wp(x_\wp)) \|
    \end{align*}
    if $\gamma_{i_1}^x \circ \cdots \circ \gamma_{i_r}^x < \operatorname{id}$ for any $r \in \mathbb N$ and $\{i_1, \ldots, i_r\} \subseteq I$, then the interconnection $(\Sigma_1, \ldots, \Sigma_\wp)$ is $\hat \omega$-input-to-state stable.
    $\lhd$
\end{theorem}
% \begin{proof}
% TODO; see \cite[Theorem~10]{Noroozi2018} with \cite[Corollary~2.6]{Sontag2022}; compare \cite[Theorem~2]{Jiang2001}.
% \end{proof}

\subsection{Sensitivity of Optimal Control}
Take any $p \in P$.
The minimizer $\bar {\mathbf u}(p)$ is the unique solution of the generalized equation \citep{Dontchev2021}
\begin{align}
    \mathcal L_p(\mathbf u) \Def= \nabla_{\mathbf u} J(\mathbf u, p) + \partial U(\mathbf u) \ni 0, \quad \mathbf u \in \mathcal U,
\end{align}
%where $N_{\mathcal U}(\mathbf u)$ is the normal cone of $\mathcal U$ at $\mathbf u$. 
where $\partial U(\mathbf u)$ is the subdifferential of $U$ at $\mathbf u$.
The inverse of \added[id=TC]{the set-valued mapping} $\mathcal L_p$ is
\begin{align*}
    \mathcal L_p^{-1}: \delta \mapsto \{ \mathbf u \in \mathcal U \, | \, \delta \in \mathcal L_p(\mathbf u) \}
\end{align*}
By strong convexity, $\mathbf u_\delta \in \mathcal L_p^{-1}(\delta)$ if and only if $\mathbf u_\delta$ is the (unique) solution of the perturbed optimization problem
\begin{align*}
    \min_{\mathbf u} J(\mathbf u, p) + U(\mathbf u) - \delta^\top \mathbf u
\end{align*}
and the mapping $\delta \mapsto \mathbf u_\delta$ is Lipschitz continuous \citep[Theorem~11.1]{Dontchev2021}. In other words, $\mathcal L_p$ is strongly regular\footnote{A set-valued map $L: \xi \rightrightarrows \zeta$ is said to be strongly regular at $\xi_0$ for $\zeta_0 \in L(\xi_0)$ if and only if $L^{-1}(\zeta)$ has a single value $\xi$ close to $\xi_0$ for all $\zeta$ around $\zeta_0$.} at $\bar {\mathbf u}(p)$ for 0 and hence, the solution map $\bar {\mathbf u}(\cdot)$ is Lipschitz continuous around $p$ \citep[Theorem~8.5]{Dontchev2021}.
Moreover, we have that
\begin{align}
    \label{eq:gradient-optimal}
    \nabla_p \bar J(p) = \left. \nabla_p J(\mathbf u, p) \right|_{\mathbf u = \bar {\mathbf u}(p)}
\end{align}
% By strong convexity of $J(\cdot)$, the optimal value function $\bar {\mathbf u}(p)$ is Lipschitz continuous in $p$. Hence, 
by virtue of \citet[Theorem~3.1]{Toan2021}.

% \comment[id=TC]{Lipschitz continuity of $\bar {\mathbf u}(\cdot)$ can be shown using \cite[Corollary~8.9]{Dontchev2021}.}

\section{Stability Analysis}
We study the dynamics of an interconnection of finite-step proximal gradient descent algorithms and derive sufficient conditions for \eqref{eq:dynopt} to be asymptotically stable with respect to $\mathcal P_\star$.
To that extend, we show that both the outer and inner optimization are $\omega$-input-to-state stable with respect to appropriate measurement functions and prove that, assuming the sufficient condition of Theorem~\ref{thm:smallgain} is satisfied, $\hat \omega$ is a size function for $\mathcal P_\star$. Moreover, we will argue that the small-gain condition can always be met if the number of inner iterations $\kappa$ is sufficiently large.

Recall that, \added[id=TC]{under moderate assumptions for the} cost function, the proximal gradient descent algorithm asymptotically converges to a fixed point if the step size is chosen to be smaller than or equal to the inverse of the Lipschitz constant of the gradient.

\begin{assumption}
    \label{ass:steplength}
    The gradients $\nabla_p \bar J(\cdot)$ and $\nabla_{\mathbf u} J(\cdot, p)$ are Lipschitz continuous (for all $p \in \mathcal P$) with constants $\nu^{-1}$ and $\mu^{-1}$, respectively, or larger.
\end{assumption}

Since $\nabla_p J(\mathbf u, \cdot)$ is smooth, $\bar {\mathbf u}(\cdot)$ is locally Lipschitz, and $\mathcal P$ is compact, as well as $\nabla_{\mathbf u} J(\cdot, p)$ being linear in $\mathbf u$, Lipschitz continuity of the gradients follows a fortiori.
Under Assumption~\ref{ass:steplength}, the solution $\{ \mathbf u_k \}_{k \geq 0}$ of \eqref{eq:dynopt-inner} for a fixed $p \in \mathcal P$ converges to $\bar {\mathbf u}(p)$ if $k \to \infty$. 
%Moreover, under some additional assumptions, if the estimated gradient $\widehat {\nabla_p J}$ corresponds to $\nabla_p \bar J$, then the solution $\{ p^\ell \}_{\ell \geq 0}$ converges to $\operatorname{null} \nabla_p \bar J$ if $\ell \to \infty$. 

We propose to estimate the gradient of $\bar J$ by the partial derivative
\begin{align}
    \label{eq:gradient-estimate}
    \widehat {\nabla_p J}(\mathbf u, p) = \nabla_p J(\mathbf u, p)
\end{align}
The interconnection now deviates from the nominal in two points: The parameter of the inner iteration changes during the optimization and the gradient for the outer iteration is perturbed by the error between $\mathbf u_\kappa^\ell$ and $\bar {\mathbf u}(p^\ell)$.

\subsection{Parametrized Gradient Descent}
Input-to-state stability of parametrized optimization algorithms has previously been studied, e.g., by \cite{Liao-McPherson2020a}. That work considered the change of parameter and optimization error for input and state, respectively. We are going to generalize the result to obtain $\omega$-input-to-state stability for the proximal gradient descent. We start by rewriting the $\kappa$-step update by \eqref{eq:dynopt-inner} to
\begin{subequations}
    \label{eq:pgd-parametrized}
\begin{align}
    \mathbf u^{\ell+1} &= g_\kappa(\mathbf u^\ell, \tilde p_\ell, \Delta p_\ell) \\
    \tilde p_{\ell+1} &= \tilde p_\ell + \Delta p_\ell
\end{align}
\end{subequations}
where $g_\kappa: \mathbb R^{Nm} \times \mathbb R^l \times \mathbb R^l \to \mathbb R^{Nm}$ is obtained by repeating the proximal gradient descent \eqref{eq:dynopt-inner} for $\kappa$ times with fixed parameter $p^\ell = \tilde p_\ell + \Delta p_\ell$. Here, the additional state $\tilde p \in \mathbb R^l$ represents the previous parameter, that is, $\tilde p_\ell = p_{\ell-1}$ for all $\ell > 0$. We introduce the measurement function
\begin{align*}
    \omega_{\mathbf u}: (\mathbf u, \tilde p) \mapsto \| \mathbf u - \bar {\mathbf u}(\tilde p) \|_2
\end{align*}
for the augmented state $(\mathbf u, \tilde p)$ of \eqref{eq:pgd-parametrized}.
Since the proximity operator is nonexpansive (Property~\ref{prop:prox-nonexp}), any solution $\{\mathbf u^\ell_k\}_{k \geq 0}$ of \eqref{eq:dynopt-inner} satisfies the contraction property
\begin{align}
    \label{eq:pgd-contraction}
    \omega_{\mathbf u}(\mathbf u^\ell_k, p^\ell) \leq \eta^k \omega_{\mathbf u}(\mathbf u^\ell_0, p^\ell)
\end{align}
for any $k \geq 0$, where the convergence rate $\eta \in (0, 1)$ follows from strong convexity of $J(\cdot, p)$ and Assumption~\ref{ass:steplength}.
We obtain the following result.

\begin{lemma}
    \label{lem:dissipation-inner}
    Let $\mathbf u \in \mathbb R^{Nm}$ and $\tilde p, \Delta p \in \mathbb R^l$; the dynamics in \eqref{eq:pgd-parametrized} satisfy the dissipation-type inequality
    \begin{align*}
        \omega_{\mathbf u}(\mathbf u_+, \tilde p_+) &\leq \eta^\kappa \omega_{\mathbf u}(\mathbf u, \tilde p) + \lambda_\star \eta^\kappa | \Delta p|
    \end{align*}
    if $\mathbf u_+ = g_\kappa(\mathbf u, \tilde p, \Delta p)$ and $\tilde p_+ = \tilde p + \Delta p$, where $\lambda_\star \geq 0$ is the Lipschitz constant of the optimal solution $\bar {\mathbf u}(\cdot)$ on $\mathcal P$.
\end{lemma}
\begin{proof}
    From the triangle inequality, we obtain
    \begin{multline}
        \label{eq:pgd-triangle}
        \omega_{\mathbf u}(\mathbf u, \tilde p + \Delta p) = \| \mathbf u - \bar {\mathbf u}(\tilde p + \Delta p) \|_2 \\ 
        \leq \| \mathbf u - \bar {\mathbf u}(\tilde p) \|_2 + \| \bar {\mathbf u}(\tilde p + \Delta p) - \bar {\mathbf u}(\tilde p) \|_2 \\
        \leq \omega_{\mathbf u}(\mathbf u, \tilde p) + \lambda_\star | \Delta p |
    \end{multline}
    and combining \eqref{eq:pgd-contraction} and \eqref{eq:pgd-triangle} yields the desired result.
\end{proof}

It remains to prove that $\omega_{\mathbf u}$ is an $\omega_{\mathbf u}$ISS Lyapunov function.

\begin{proposition}
    \label{prop:iss-inner}
    The $\kappa$-step update of the proximal gradient descent, as given in \eqref{eq:pgd-parametrized}, is $\omega_{\mathbf u}$-input-to-state stable with gain $\gamma_\kappa$; and $\gamma_\kappa' \to 0$ uniformly as $\kappa \to \infty$.
    % and $\omega_{\mathbf u}$ISS Lyapunov function $V_{\mathbf u} = (\omega_{\mathbf u})^2$.
\end{proposition}
\begin{proof}
    Let $\mathbf u_+ = g_\kappa(\mathbf u, \tilde p, \Delta p)$ and $\tilde p_+ = \tilde p + \Delta p$ for some $\mathbf u \in \mathbb R^{Nm}$ and $\tilde p, \Delta p \in \mathbb R^l$ and
    choose $\epsilon \in (0, 1)$; by virtue of Lemma~\ref{lem:dissipation-inner}, we have that
    \begin{align*}
        \omega_{\mathbf u}(\mathbf u_+, \tilde p_+) \leq (1 - \epsilon) \omega_{\mathbf u}(\mathbf u, \tilde p) \Def= \alpha(\omega_{\mathbf u}(\mathbf u, \tilde p))
    \end{align*}
    if $(1 - \eta^\kappa - \epsilon) \omega_{\mathbf u}(\mathbf u, \tilde p) \geq \lambda_\star {\eta^\kappa} | \Delta p |$; and
    \begin{align*}
        \omega_{\mathbf u}(\mathbf u_+, \tilde p_+) < \lambda_\star \frac{\eta^\kappa + 1}{1 - \eta^\kappa - \epsilon} \eta^\kappa | \Delta p | \Def= \gamma_\kappa(|\Delta p|)
    \end{align*}
    otherwise. If $\epsilon$ is sufficiently small, we have that $\alpha, \gamma_\kappa \in \mathcal K_\infty$ as well as
    \begin{align*}
        \omega_{\mathbf u}(\mathbf u_+, \tilde p_+) \leq \max \{ \alpha(\omega_{\mathbf u}(\mathbf u, \tilde p)), \gamma_\kappa(|\Delta p|) \}
    \end{align*}
    proving that $\omega_{\mathbf u}$ is an $\omega_{\mathbf u}$ISS Lyapunov function. Moreover, $\gamma_\kappa$ is linear with $\gamma_\kappa' \to 0$ as $\kappa \to \infty$ by l'Hôpital's rule.
\end{proof}

In other words, \eqref{eq:pgd-parametrized} is input-to-output stable, with input $\Delta p$ and output $\mathbf u$, independently of the state $\tilde p$ -- that is, the actual parameter $p$. 

\begin{remark}
    $\omega_{\mathbf u}$-input-to-state stability of \eqref{eq:pgd-parametrized} can also be viewed as {\em parametrized} input-to-state stability of the original proximal gradient descent algorithm. In contrast, \citet{Liao-McPherson2020a} proved input-to-state stability with both $p$ and $\Delta p$ as inputs (although the gain on $p$ was zero).
\end{remark}

\subsection{Perturbed Gradient Descent}
In recent work, \citet{Sontag2022} established input-to-state stability of a perturbed steepest descent algorithm with respect to the set of minimizers $\mathcal P_\star$, under the assumptions that \added[id=IK]{the cost function and norm of the gradient} are size functions for $\mathcal P_\star$. Convergence of proximal gradient schemes under perturbed gradients was studied in earlier works \citep{Lemaire1988, Tossings1994} and more recently by \citet{Atchade2017}.

Building upon those previous works, we prove input-to-state stability of the proximal gradient descent scheme with perturbed gradient, as given in \eqref{eq:dynopt-outer}. We start with the simplified iteration
\begin{align}
    \label{eq:pgd-perturbation}
    p^{\ell+1} = T_\nu(p^\ell, d) \Def= \operatorname {prox}_P \big[ p^\ell - \nu (\nabla \bar J(p^\ell)^\top + d) \big]
\end{align}
where $d \in \mathbb R^l$ is a perturbation for the gradient. Adopting from \citet{Sontag2022}, we make the following assumptions.

\begin{assumption}
    \label{ass:iss-sufficient}
    The cost function $\bar J$ satisfies:
    \begin{itemize}
        \item the function $J_\star: p \mapsto (\bar J + P)(p) - j^\star$ is a size function for $\mathcal P_\star$, where $j^\star$ is the global minimum of $\bar J + P$; %on $\mathcal P$;
        % \item the gradient $\nabla \bar J$ vanishes exactly on $\mathcal P_\star$;
        \item the function $\Lambda: p \mapsto |T_\nu(p, 0) - p|$ is a size function for $\mathcal P_\star$.
    \end{itemize}
\end{assumption}

\added[id=TC]{%
Since $P$ has a compact domain, $J_\star$ is a size function if and only if $\bar J + P$ is positive definite with respect to $\mathcal P_\star$. Moreover, the assumption on $\Lambda$ is equivalent to the necessary conditions for optimality being sufficient, too.
}
% Since $\mathcal P$ is compact, we can always construct a size function $\tilde J$ that coincides with $J_\star$ on $P$ provided that the latter is positive definite (on $\mathcal P$) with respect to $\mathcal P_\star$. The assumption on $\Lambda$ implies that the necessary conditions for optimality are sufficient, too. 
%If $\mathcal P = \mathbb R^l$,
If $P \equiv 0$, as for gradient descent, $\Lambda(p)$ corresponds to $|\nabla \bar J(p)|$ for all $p \in \mathbb R^l$.

\begin{theorem}
    \label{thm:iss-outer}
    Let $\omega_p: \mathbb R^l \to \mathbb R_{\geq 0}$ be a size function for $\mathcal P_\star$; the dynamics in \eqref{eq:pgd-perturbation} are $\omega_p$-input-to-state stable.
\end{theorem}
\begin{proof}
    \begin{subequations}
    Take $p, d \in \mathbb R^l$; since $T_\nu(p, d) \in \mathcal P$, we have that
    \begin{multline}
        \label{eq:pgd-proxvalue-1}
        (\bar J + P)(T_\nu(p, d)) - (\bar J + P)(p) \\
        \leq -(2\nu)^{-1} | T_\nu(p, d) - p |^2 -\alpha^{-1} \langle T_\nu(p, d) - p, -\nu d \rangle
    \end{multline}
    by Lemma~\ref{lem:prox-value}. Rewriting $T_\nu(p,d) - p$ and applying the triangle inequality to both sides we obtain
    \begin{align*}
        \big| \Lambda(p) - \nu |d| \big|
        % &\leq |T_\nu(p,d) - T_\nu(p,0)| - |T_\nu(p,0) - p| \\
        \leq | T_\nu(p,d) - p |
        % &\leq |T_\nu(p,d) - T_\nu(p,0)| + |T_\nu(p,0) - p| \\ 
        \leq \Lambda(p) + \nu |d|
    \end{align*}
    and thus, \eqref{eq:pgd-proxvalue-1} yields
    \begin{align}
        \label{eq:pgd-proxvalue-2}
        J_\star(T_\nu(p,d)) - J_\star(p) \leq -(2\nu)^{-1} \Lambda(p)^2 + 2 \Lambda(p) |d| + \nu |d|^2
    \end{align}
    where rewriting the left-hand side is trivial. We now argue with \citet[Proof of Theorem~4]{Sontag2022} that, since $J_\star$ and $\Lambda$ both are size functions for $\mathcal P_\star$, there exists $\mathcal K_\infty$ functions $b_1$ and $b_2$ satisfying
    \begin{align*}
        b_1 J_\star(p) \leq \Lambda(p) \leq b_2 J_\star(p)
    \end{align*}
    for all $p \in \mathbb R^l$, where linearity of $b_1$ and $b_2$ can be assumed without loss of generality for $\mathcal P$ is compact. 
    
    Inserting the inequalities into \eqref{eq:pgd-proxvalue-2}, we have that, for all $p,d \in \mathbb R^l$,
    \begin{align}
        J_\star(T_\nu(p,d)) - J_\star(p) \leq -c J_\star(p)^2 + 2 b_2 J_\star(p) |d| + \nu |d|^2
    \end{align}
    for some $c > 0$ satisfying $c \leq b_1^2 (2\nu)^{-1}$. Choose $\chi > 0$ such that, for all $p,d \in \mathbb R^l$,
    \begin{align*}
        J_\star(T_\nu(p,d)) \leq J_\star(p) -\rho J_\star(p)^2
    \end{align*}
    if $J_\star(p) \geq \chi |d|$, where $\rho > 0$; and otherwise,
    \begin{align*}
        J_\star(T_\nu(p,d)) \leq \chi |d| + \chi_0 |d|^2
    \end{align*}
    with $\chi_0 > 0$. 
    \end{subequations}
    Hence, as $J_\star \circ T_\nu$ is upper-bounded, there exists $\alpha \in (0, 1)$ and $\gamma > 0$ satisfying
    \begin{align*}
        J_\star(T_\nu(p,d)) \leq \max\{ \alpha J_\star(p), \gamma |d| \}
    \end{align*}
    for all $p, d \in \mathbb R^l$, proving that $J_\star$ is an $\omega_p$ISS-Lyapunov function for \eqref{eq:pgd-perturbation}. This concludes the proof.
\end{proof}

It rests to prove input-to-state stability if the gradient perturbation is caused by the estimation error of \eqref{eq:gradient-estimate}. From \eqref{eq:gradient-optimal}, we have that the $i$-th component of the gradient error vector satisfies
\begin{multline}
    \big| \big[\nabla_p \bar J - \widehat {\nabla_p J}\big]_i (p, \mathbf u) \big| \leq (\bar {\mathbf u}(p) - \mathbf u)^\top \nabla_{p_i} H(p) (\bar {\mathbf u}(p) - \mathbf u) \\
    + |\partial_{p_i} g(p)^\top (\bar {\mathbf u}(p) - \mathbf u)|
\end{multline}
where $\nabla_{p_i}$ denotes the element-wise partial derivative with respect to the $i$-th component of $p \in \mathbb R^l$. Since $\mathcal U$ is compact, there exists $\digamma > 0$ such that
\begin{align*}
    \big| \big[\nabla_p \bar J - \widehat {\nabla_p J}\big] (p, \mathbf u) \big| \leq \digamma | \bar {\mathbf u}(p) - \mathbf u |
\end{align*}
for all $(p, \mathbf u) \in \mathcal P \times \mathcal U$. We conclude with the following result.

\begin{corollary}
    % Suppose $\mathcal U$ is compact; then 
    The dynamics in \eqref{eq:dynopt-outer} are $\omega_p$-input-to-state stable with respect to the input $\bar {\mathbf u}(p) - \mathbf u$ for any size function $\omega_p$ for $\mathcal P_\star$ on $\mathbb R^l$.
    $\lhd$
\end{corollary}

\subsection{Interconnected Gradient Descent}
We are going to prove asymptotic stability of the bilevel proximal gradient scheme by application of the small-gain theorem. 

% \begin{assumption}
%     The set $\mathcal U$ is compact.
% \end{assumption}

Suppose $\{(p^\ell, \mathbf u^\ell)\}_{\ell \geq 0} \subset P \times \mathcal U$ is a solution to the interconnection of \eqref{eq:dynopt} and \eqref{eq:dynopt-inner}; for any $\ell > 0$, take $\tilde p_{\ell+1} = p^\ell$ and $\Delta p_\ell = p^\ell - p^{\ell-1}$ and recall that
\begin{subequations}
    \label{eq:pgd-gains}
\begin{align}
    \omega_{\mathbf u}(\mathbf u^{\ell+1}, \tilde p_{\ell+1}) &\leq \max\{ \alpha_\kappa \omega_{\mathbf u}(\mathbf u^\ell, \tilde p^\ell), \gamma_\kappa | \Delta p^\ell | \} \\
    J_\star(\tilde p^{\ell+1}) &\leq \max\{ \alpha_0 J_\star(\tilde p^\ell), \gamma_0 \digamma \omega_{\mathbf u}(\mathbf u^\ell, \tilde p_\ell) \}
\end{align}
where $\alpha_0, \alpha_\kappa \in (0, 1)$, $\gamma_0, \gamma_\kappa, \digamma > 0$, and $\omega_{\mathbf u}(\mathbf u, p) \equiv | \bar {\mathbf u}(p) - \mathbf u|$ have been obtained beforehand,
as well as
\begin{align}
    | \Delta p_{\ell} | \leq \Lambda(\tilde p_{\ell}) + \nu \digamma \omega_{\mathbf u}(\mathbf u^{\ell}, \tilde p_{\ell})
\end{align}
by Lemma~\ref{lem:prox-perturbation}.
\end{subequations}

\begin{proposition}
	There exists $K \in \mathbb N$ such that, for all $\kappa \geq K$,
%    If $\alpha \digamma \lambda_\star \eta^\kappa \in (0, 1)$, then 
    the interconnection of \eqref{eq:dynopt} and \eqref{eq:dynopt-inner} is asymptotically stable with respect to the set
    \begin{align*}
        \Omega = \{ (\bar {\mathbf u}(p), p) \, | \, p \in \mathcal P_\star \}
    \end{align*}
    and $| \Delta p_\ell | \to 0$ as $\ell \to \infty$.
\end{proposition}
\begin{proof}
%     Let $\tilde p_{\ell+1} = p^\ell$ for all $\ell \geq 0$; then 
% \begin{align*}
%     \omega_{\mathbf u}(\mathbf u^{\ell+1}, \tilde p_{\ell+1}) &\leq \eta^\kappa \omega_{\mathbf u}(\mathbf u^\ell, \tilde p_\ell) + \lambda_\star \eta^\kappa | \tilde p_{\ell+1} - \tilde p_\ell | \\
%     J_\star(\tilde p_{\ell+1}) &\leq -c J_\star(\tilde p_\ell)^2 + (2b_2 + 1) J_\star(\tilde p_{\ell}) \digamma \omega_{\mathbf u}(\mathbf u^{\ell}, \tilde p_{\ell}) \\ &\quad + \nu \digamma^2 \omega_{\mathbf u}(\mathbf u^{\ell}, \tilde p_{\ell})^2
% \end{align*}
    Take $b > 0$ such that $\Lambda(p) \leq b J_\star(p)$ for all $p \in \mathbb R^l$; from \eqref{eq:pgd-gains} and Lemma~\ref{lem:dissipation-inner} we obtain
\begin{multline*}
    | \Delta p_{\ell+1} | \leq \Lambda(\tilde p_{\ell+1}) + \nu \digamma \omega_{\mathbf u}(\mathbf u^{\ell+1}, \tilde p_{\ell+1}) \\
\begin{aligned}
    &\leq \nu \digamma \lambda_\star \eta^\kappa | \Delta p_\ell | + \nu \digamma \eta^\kappa \omega_{\mathbf u}(\mathbf u^\ell, \tilde p_\ell) \\ 
    &\hspace{5em} + b \max\{ \alpha_0 J_\star(\tilde p_\ell), \gamma_0 \digamma \omega_{\mathbf u}(\mathbf u^\ell, \tilde p_\ell) \} \\
    %
%    &= \nu \digamma \lambda_\star \eta^\kappa | \Delta p_\ell | + \max\big\{ b \alpha_0 J_\star(\tilde p_\ell) \\
%    &\quad + \nu \digamma \eta^\kappa \omega_{\mathbf u}(\mathbf u^\ell, \tilde p_\ell), (\nu \eta^\kappa + b \gamma_0) \digamma \omega_{\mathbf u}(\mathbf u^\ell, \tilde p_\ell) \big\} \\
    %
    &\leq \nu \digamma \lambda_\star \eta^\kappa | \Delta p_\ell | + \max\big\{ (1 + \nu \eta^\kappa (b \gamma_0)^{-1}) b \alpha_0 J_\star(\tilde p_\ell),
\end{aligned}
    \\
    (\nu \eta^\kappa + b \gamma_0) \digamma \omega_{\mathbf u}(\mathbf u^\ell, \tilde p_\ell) \big\}
\end{multline*}
If $K$ is sufficiently large, then $\nu \digamma \lambda_\star \eta^K \in (0,1)$ and, repeating the argument in Proposition~\ref{prop:iss-inner}, we obtain $\varrho \in (0, 1)$ and $\vartheta > 0$ satisfying, for all $\kappa \geq K$,
\begin{align}
    \label{eq:pgd-increment}
    | \Delta p_{\ell+1} | \leq \max\{ \varrho | \Delta p_\ell |, \gamma_1 J_\star(\tilde p_\ell), \gamma_2 \omega_{\mathbf u}(\mathbf u^\ell, \tilde p_\ell) \}
\end{align}
with $\gamma_1 = \vartheta (b + \nu \eta^\kappa {\gamma_0}^{-1}) \alpha_0$ and $\gamma_2 = \vartheta (\nu \eta^\kappa + b \gamma_0) \digamma$.

Combining \eqref{eq:pgd-gains} with \eqref{eq:pgd-increment} yields an interconnection of three $\omega$ISS-Lyapunov functions for $\mathbf u$, $\tilde p$, and $\Delta p$. Clearly, if $K$ is sufficiently large, then
\begin{gather*}
	\gamma_2 \cdot \gamma_\kappa < 1 \\
	\gamma_0 \cdot \gamma_1 \cdot \gamma_\kappa < 1
\end{gather*}
and by virtue of Theorem~\ref{thm:smallgain}, the interconnection of \eqref{eq:dynopt} and \eqref{eq:dynopt-inner} is $\hat \omega$-input-to-state stable, where
\begin{align*}
	\hat \omega: (\tilde p, \mathbf u, \Delta p) = \| (\omega_p(\tilde p), \omega_{\mathbf u}(\mathbf u, \tilde p), |\Delta p|) \|
\end{align*}
for all $\kappa \geq K$. We claim that $\hat \omega$ is a size function for $\Omega \times \{0\}$: Take $(\tilde p, \mathbf u, \Delta p) \in \mathbb R^l \times \mathbb R^{Nm} \times \mathbb R^l$, then $\hat \omega(\tilde p, \mathbf u, \Delta p) = 0$ if and only if $\tilde p \in \mathcal P_\star$, $\mathbf u = \bar {\mathbf u}(\tilde p)$, and $|\Delta p| = 0$. Suppose $\Pi \subset \mathbb R^l$ is compact; then, for any $(\tilde p, \mathbf u) \in \Pi \times \mathbb R^{Nm}$,
\begin{align*}
	\omega_{\mathbf u}(\mathbf u, \tilde p) \geq | \mathbf u | - | \bar {\mathbf u}(\tilde p) | \geq | \mathbf u | - \upsilon_\Pi
\end{align*}
where $\upsilon_\Pi = \max_{p \in \Pi} | \bar {\mathbf u}(p) | < \infty$ by Lipschitz continuity of $\bar {\mathbf u}(\cdot)$ on $\Pi$. Hence, since $\omega_p$ is a size function and $\| \cdot \|$ is monotone, $\hat \omega(\tilde p, \mathbf u, \Delta p) \to \infty$ as either $| p | \to \infty$, $| \mathbf u | \to \infty$, or $| \Delta p | \to \infty$.

We conclude that $(\tilde p_\ell, \mathbf u^\ell) \to \Omega$ and $| \Delta p_\ell | \to 0$ as $\ell \to \infty$. Noting that $\Delta p_\ell = p^\ell - \tilde p_\ell$, this is the desired result.
\end{proof}

\section{Special Case}
We have shown that the bilevel optimization scheme converges asymptotically to the optimal solution even if the inner-loop problem is solved inexactly.
However, specific conditions for the optimization problem of \eqref{eq:linear-system} and \eqref{eq:quadratic-cost} to satisfy Assumption~\ref{ass:iss-sufficient} remain an open research question. In order to illustrate our theoretic results we consider the case that the parameters only enters through an additive linear term, viz.
\begin{align}
	x_{k+1} = A x_k + B u_k + E p
\end{align}
and the quadratic cost is constant in the parameter. In this scenario, the parameter may correspond to the initial condition of the optimal control problem or an additional input that is constant over the control horizon. The inner-loop cost function now simplifies to
\begin{align}
	J(\mathbf u, p) = (\mathbf u, p)^\top H (\mathbf u, p) + g^\top (\mathbf u, p) + c
\end{align}
where the block matrix $H$ is positive semidefinite by positive definiteness of $Q$ and $R$.
% where $H$ is a block matrix with diagonal elements $H_{uu} \in \mathbb R^{Nm \times Nm}$ and $H_{pp} \in \mathbb R^{l \times l}$ as well as off-diagonal elements $H_{up} = H_{pu}^\top \in \mathbb R^{Nm \times l}$. Since $H_{uu}$ is positive definite and $H_{pp}$ is positive semidefinite by positive definiteness of $Q$ and $R$, respectively, the cost $J(\mathbf u, p)$ is strongly convex in $\mathbf u$ and convex $p$. The optimal solution $\bar {\mathbf u}(p)$ then is stationary point of
% \begin{align*}
% 	\mathbf u = \operatorname{prox}_U\big[\mathbf u - H_{uu} \mathbf u - H_{up} p - g_{u} \big]
% \end{align*}
Hence, the cost $J$ is convex in $(\mathbf u, p)$ and, since $U$ is convex, $\bar J(p) = \min_{\mathbf u} J(\mathbf u, p) + U(\mathbf u)$ is convex too. Consequently, $\mathcal P_\star$ is convex and the necessary conditions for optimality are sufficient, that is, $\Lambda(p) = 0$ if and only if $p \in \mathcal P_\star$. Assumption~\ref{ass:iss-sufficient} thus holds since $J_\star(p) \geq P(p)$ and $\Lambda(p) \geq \operatorname {dist}_{\mathcal P}(p)$ diverge for $|p| \to \infty$. Moreover, the gradient of $J$ is linear in $\mathbf u$ and $p$, viz.
\[
    \nabla_p J(\mathbf u, p) = (\mathbf u, p)^\top H_{p} + g_p^\top
\]
and its Lipschitz constant is equal to $\| H_p \|$.
\added[id=TC]{%
We conclude from the results of the previous section that the bilevel optimization is asymptotically stable if the number of inner-loop iterations $\kappa$ is sufficiently large.%
}

% \newpage
\section{Conclusion}
We have shown that convergence conditions for a class of bilevel optimization problems with inexactly implemented proximal gradient optimization can be derived using input-to-state stability arguments.  Future work will focus on broadening the class of problems and optimization algorithms for which similar approaches can be followed as well as on applications to real world problems motivating this work.

% \begin{ack}
% Ilya Kolmanovsky acknowledges support by United States Air Force Office of Scientific Research Grant number FA9550-20-1-0385.
% \end{ack}

\bibliography{main,kolmanovsky}             % bib file to produce the bibliography

\begin{thebibliography}{19}
\providecommand{\natexlab}[1]{#1}
\providecommand{\url}[1]{\texttt{#1}}
\providecommand{\urlprefix}{URL }
\expandafter\ifx\csname urlstyle\endcsname\relax
  \providecommand{\doi}[1]{doi:\discretionary{}{}{}#1}\else
  \providecommand{\doi}{doi:\discretionary{}{}{}\begingroup
  \urlstyle{rm}\Url}\fi

\bibitem[{Atchad{\'{e}} et~al.(2017)Atchad{\'{e}}, Fort, and
  Moulines}]{Atchade2017}
Atchad{\'{e}}, Y.F., Fort, G., and Moulines, E. (2017).
\newblock On perturbed proximal gradient algorithms.
\newblock \emph{Journal of Machine Learning Research}, 18, 1--33.

\bibitem[{Basilico et~al.(2017)Basilico, Coniglio, Gatti, and
  Marchesi}]{basilico2017bilevel}
Basilico, N., Coniglio, S., Gatti, N., and Marchesi, A. (2017).
\newblock Bilevel programming approaches to the computation of optimistic and
  pessimistic single-leader-multi-follower equilibria.
\newblock In \emph{16th International Symposium on Experimental Algorithms},
  volume~75, 1--14.

\bibitem[{Bennett et~al.(2008)Bennett, Kunapuli, Hu, and
  Pang}]{bennett2008bilevel}
Bennett, K.P., Kunapuli, G., Hu, J., and Pang, J.S. (2008).
\newblock Bilevel optimization and machine learning.
\newblock In \emph{IEEE World Congress on Computational Intelligence}, 25--47.
  Springer.

\bibitem[{Colson et~al.(2007)Colson, Marcotte, and Savard}]{colson2007overview}
Colson, B., Marcotte, P., and Savard, G. (2007).
\newblock An overview of bilevel optimization.
\newblock \emph{Annals of operations research}, 153(1), 235--256.

\bibitem[{Cunis et~al.(2022)Cunis, Kolmanovsky, and Cesnik}]{cunis2022control}
Cunis, T., Kolmanovsky, I.V., and Cesnik, C.E. (2022).
\newblock Control co-design optimization: Integrating nonlinear controllability
  into a multidisciplinary design process.
\newblock In \emph{AIAA Scitech 2022 Forum}, 2176.

\bibitem[{Dontchev(2021)}]{Dontchev2021}
Dontchev, A.L. (2021).
\newblock \emph{Lectures on Variational Analysis}.
\newblock Number 205 in Applied Mathematical Sciences. Springer, Cham.
\newblock \doi{10.1007/978-3-030-79911-3}.

\bibitem[{Franceschi et~al.(2018)Franceschi, Frasconi, Salzo, Grazzi, and
  Pontil}]{franceschi2018bilevel}
Franceschi, L., Frasconi, P., Salzo, S., Grazzi, R., and Pontil, M. (2018).
\newblock Bilevel programming for hyperparameter optimization and
  meta-learning.
\newblock In \emph{International Conference on Machine Learning}, 1568--1577.

\bibitem[{Garcia-Sanz(2019)}]{garcia2019control}
Garcia-Sanz, M. (2019).
\newblock Control co-design: an engineering game changer.
\newblock \emph{Advanced Control for Applications: Engineering and Industrial
  Systems}, 1(1), e18.

\bibitem[{Garone et~al.(2017)Garone, Di~Cairano, and
  Kolmanovsky}]{garone2017reference}
Garone, E., Di~Cairano, S., and Kolmanovsky, I. (2017).
\newblock Reference and command governors for systems with constraints: A
  survey on theory and applications.
\newblock \emph{Automatica}, 75, 306--328.

\bibitem[{Kellett and Dower(2012)}]{Kellett2012}
Kellett, C.M. and Dower, P.M. (2012).
\newblock A generalization of input-to-state stability.
\newblock In \emph{51st IEEE Conference on Decision and Control}, 2970--2975.
  IEEE, Maui, HI.
\newblock \doi{10.1109/CDC.2012.6426008}.

\bibitem[{Lemaire(1988)}]{Lemaire1988}
Lemaire, B. (1988).
\newblock Coupling optimization methods and variational convergence.
\newblock In K.H. Hoffmann, J.~Zowe, J.B. Hiriart-Urruty, and C.~Lemarechal
  (eds.), \emph{Trends in Mathematical Optimization}, volume~84 of
  \emph{International Series of Numerical Mathematics}, 163--179.
  Birkh{\"{a}}user, Basel.
\newblock \doi{10.1007/978-3-0348-9297-1_12}.

\bibitem[{Liao-McPherson et~al.(2020)Liao-McPherson, Nicotra, and
  Kolmanovsky}]{Liao-McPherson2020a}
Liao-McPherson, D., Nicotra, M.M., and Kolmanovsky, I. (2020).
\newblock Time-distributed optimization for real-time model predictive control:
  Stability, robustness, and constraint satisfaction.
\newblock \emph{Automatica}, 117, 108973.
\newblock \doi{10.1016/j.automatica.2020.108973}.

\bibitem[{Noroozi et~al.(2018)Noroozi, Geiselhart, Grune, Ruffer, and
  Wirth}]{Noroozi2018}
Noroozi, N., Geiselhart, R., Grune, L., Ruffer, B.S., and Wirth, F.R. (2018).
\newblock Nonconservative discrete-time iss small-gain conditions for closed
  sets.
\newblock \emph{IEEE Transactions on Automatic Control}, 63(5), 1231--1242.
\newblock \doi{10.1109/TAC.2017.2735194}.

\bibitem[{Rockafellar(1976)}]{Rockafellar1976}
Rockafellar, R.T. (1976).
\newblock Monotone operators and the proximal point algorithm.
\newblock \emph{SIAM Journal on Control and Optimization}, 14(5), 877--898.
\newblock \doi{10.1137/0314056}.

\bibitem[{Sontag(1995)}]{Sontag1995a}
Sontag, E.D. (1995).
\newblock On the input-to-state stability property.
\newblock \emph{European Journal of Control}, 1(24), 24--36.
\newblock \doi{10.1016/S0947-3580(95)70005-X}.

\bibitem[{Sontag(2022)}]{Sontag2022}
Sontag, E.D. (2022).
\newblock Remarks on input to state stability of perturbed gradient flows,
  motivated by model-free feedback control learning.
\newblock \emph{Systems and Control Letters}, 161, 105138.
\newblock \doi{10.1016/j.sysconle.2022.105138}.

\bibitem[{Sontag and Wang(1995)}]{Sontag1995b}
Sontag, E.D. and Wang, Y. (1995).
\newblock On characterizations of input-to-state stability with respect to
  compact sets.
\newblock In \emph{3rd IFAC Symposium on Nonlinear Control Systems Design},
  203--208. Elsevier, Tahoe City, CA.
\newblock \doi{10.1016/s1474-6670(17)46831-5}.

\bibitem[{Toan(2021)}]{Toan2021}
Toan, N.T. (2021).
\newblock Differential stability of discrete optimal control problems with
  mixed contraints.
\newblock \emph{Positivity}.
\newblock \doi{10.1007/s11117-021-00812-x}.

\bibitem[{Tossings(1994)}]{Tossings1994}
Tossings, P. (1994).
\newblock The perturbed proximal point algorithm and some of its applications.
\newblock \emph{Applied Mathematics \& Optimization}, 29, 125--159.
\newblock \doi{10.1007/BF01204180}.

\end{thebibliography}
                                                     % with bibtex (preferred)

\appendix
\section*{Appendix}
Let $f: \mathbb R^n \to \overline {\mathbb R}$ be a proper convex and lower semicontinuous function; assume that the function $g: \mathbb R^n \to \mathbb R$ is differentiable and that its gradient $\nabla g$ is Lipschitz continuous on the domain of $f$ with constant $\lambda \geq 0$. Define
\begin{align*}
    h(\xi) = f(\xi) + g(\xi)
\end{align*}
for all $\xi \in \mathbb R^n$.
The following results are based on \citet[Lemmas 8--10]{Atchade2017}.

\begin{lemma}
    \label{lem:prox-constraint}
    Take $\nu > 0$; then
    \begin{multline*}
        f(\operatorname{prox}_{\nu f}(\xi)) - f(\xi') \\ \leq -\nu^{-1} \langle \operatorname{prox}_{\nu f}(\xi) - \xi', \operatorname{prox}_{\nu f}(\xi) - \xi \rangle
    \end{multline*}
    for all $\xi, \xi' \in \mathbb R^n$.
    $\lhd$
\end{lemma}

\begin{lemma}
    \label{lem:prox-value}
    Let $\nu > 0$ satisfy $\nu \leq \lambda^{-1}$; then
    \begin{multline*}
        h(\operatorname{prox}_{\nu f}(\xi)) - h(\xi') \leq -(2\nu)^{-1} | \operatorname{prox}_{\nu f}(\xi) - \xi' |^2 \\
        - \nu^{-1} \langle \operatorname{prox}_{\nu f}(\xi) - \xi', \xi' - \nu \nabla f(\xi') - \xi \rangle
    \end{multline*}
    for all $\xi, \xi' \in \mathbb R^n$.
\end{lemma}
\begin{proof}
    By Lipschitz continuity of $\nabla g$, we have that
    \[
        g(\zeta) - g(\xi') \leq \langle \nabla g(\xi'), \zeta - \xi' \rangle + (2\nu)^{-1} | \zeta - \xi' |^2
    \]
    for all $\zeta, \xi' \in \mathbb R^n$. Setting $\zeta = \operatorname{prox}_{\nu f}(\xi)$ for any $\xi \in \mathbb R^n$ and applying Lemma~\ref{lem:prox-constraint} yields the desired result.
\end{proof}

\begin{lemma}
    \label{lem:prox-perturbation}
    Take $\nu > 0$ and $\eta \in \mathbb R^n$; then
    \begin{align*}
        | \operatorname{prox}_{\nu f}(\xi + \eta) - \operatorname{prox}_{\nu f}(\xi) | \leq | \eta |
    \end{align*}
    for all $\xi \in \mathbb R^n$.
\end{lemma}
\begin{proof}
    Follows immediately from Property~\ref{prop:prox-nonexp}.
\end{proof}

\end{document}